\documentclass[10pt,a4paper]{article}
\usepackage{amsmath,amsxtra,amssymb,latexsym,amscd,amsfonts,multicol,enumerate,ifthen,stmaryrd,indentfirst,amsthm,amstext}
\usepackage[mathscr]{eucal}
\usepackage{graphics, graphpap}
\usepackage[pdftex]{graphicx}
\usepackage{xspace}
\usepackage{array, tabularx, longtable}
\usepackage{multicol,color}
\usepackage{mathrsfs}
\usepackage{makeidx}
\usepackage{indentfirst}
\usepackage{subfigure}

\usepackage{hyperref}
\usepackage[T1]{fontenc}
\belowdisplayshortskip =\abovedisplayshortskip
\def\NoBlackBoxes{\overfullrule=0pt }
\NoBlackBoxes
\theoremstyle{plain}
\newtheorem{theorem}{Theorem}
\newtheorem{lemma}{Lemma}

\newtheorem{corollary}{Corollary}
\newcommand\R{{\mathbb R}}
\newcommand\E{{\textnormal E}}
\newcommand\PP{{\textnormal P}}

\begin{document}
\begin{center}

\textbf{ON THE RATE OF CONVERGENCE FOR CENTRAL LIMIT THEOREMS OF SOJOURN TIMES OF GAUSSIAN FIELDS}

\end{center}
\begin{center}

VIET-HUNG PHAM
\\Institut de math\'{e}matiques de Toulouse \\ Universit\'{e} Paul Sabatier (Toulouse III) \\ 118, route de Narbonne \\ 31062 TOULOUSE Cedex 09\\FRANCE\\
pgviethung@gmail.com
\end{center}
\abstract {The aim of this paper is to control the rate of convergence for central limit theorems of sojourn times of Gaussian fields in both cases: the fixed and the moving level. Our main tools are the Malliavin calculus and the Stein's method, developed by Nualart, Peccati and Nourdin. We also extend some results of Berman to the multidimensional case.}
\section{Introduction}
Let $X=\{X(t), t\in \R^d\}$ be a stationary centered Gaussian field and $T$ be a measurable subset of $\R^d$. The sojourn time (or the volume of the excursion set) of $X$ above the level $u_T$ in $T$ is defined as
$$\int_T \mathbb{I}(X(t)\geq u_T)dt.$$
\indent The origin of this subject is the intersection between the study of the geometric properties of random surfaces and the one of the non-linear functionals of Gaussian fields. Moreover, it has many applications in statistics of random processes (see, for example, Spodarev and Timmermann\cite{6}). \\
\indent The case of a fixed level: $u_T=u=const$ has been addressed in dimension 1 by of Sun \cite{19}, Chambers and Slud \cite{7}, Major \cite{11}, and Giraitis and Surgailis \cite{8}. Later on, some multidimensional versions were proved by Breuer and Major \cite{4}, Arcones \cite{1}, Ivanov and Leonenko \cite{9} and Bulinski, Spodarev and Timmermann \cite{6}. Their works are based on the following assumption
\begin{itemize}
 \item[(A).] $\{X(t): \, t\in \R^d \}$ is a stationary centered Gaussian field with unit variance  and covariance function $\rho(t)$ such that $$\int_{\R^d}|\rho(t)|dt < \infty,$$
 \end{itemize}
and can be presented in the following statement.
\begin{theorem} \label{th1}
Let $\{X(t): \, t\in \R^d \}$ be a random field satisfying the condition (A). For a fixed real-valued $u$, define the sojourn time as
\begin{equation}\label{st}
\displaystyle S_T=\int_{[0,T]^d} \mathbb{I}(X(t)\geq u)dt.
\end{equation}
Then, as T tends to infinity, 
$$\frac{S_T-T^d\overline{\Phi}(u)}{\sqrt{T^d}} \xrightarrow{d} \mathcal{N}(0,\sigma^2),$$
where 
\begin{equation}\label{sig}
\displaystyle 0< \sigma^2=\sum_{n=1}^{\infty} \frac{\varphi^2(u)H_{n-1}^2(u)}{n!} \int_{\R^d}\rho^n(t)dt < \infty,
\end{equation}
$\varphi$ is the density function of the standard Gaussian law and $\overline{\Phi}$ is the tail of its distribution.
\end{theorem}

\indent Berman \cite{3} considered the problem for a Gaussian process in the case when the level depends on $T$. When the covariance function is not integrable, he assumed that the main component of the sojourn time is the first chaos in the Wiener chaos expansion. Else, his arguments were based on the spectral representation
$$\rho(t)=\int_{R}b(t+s)b(s)ds,$$
with the mixing condition $b\in L^1\cap L^2$ and the $m$-dependent method. More precisely, he approximated the function $b$ by a sequence of functions with compact support obtaining a family of $m$-dependent processes converges to the original one, and then he could use the central limit theorems that had been proved for this kind of process. His method can be applied in the multivariate case.\\
\indent However, the above works do not give us much information about the rate of convergence for the central limit theorems. Then, in this paper, we aim to control the speed in both cases: the fixed and the moving level. Our approaches come from the recent techniques, developed by Nualart, Peccati and Nourdin (\cite{12},\cite{13},\cite{17}, etc.), that are the combination between the Malliavin calculus and the Stein's method.  Here, we consider the Wasserstein distance for two integrable variables
$$d(X,Y)=\underset{h\in \textnormal{Lip(1)}}{\sup}|\E(h(X))-\E(h(Y))|,$$
where Lip(1) is the collection of all Lipschitz functions with Lipschitz constant $\leq 1$. Our main results are the following:
\begin{theorem}[\textbf{Fixed level}] \label{th2}
Let $\{X(t): \, t\in \R^d \}$ be a random field satisfying the condition (A). Assume that the covariance function $\rho$ satisfies
\begin{equation}\label{dk}
\displaystyle \int_{\R^d\setminus[-a,a]^d}|\rho(t)|dt \leq (const)(\log a)^{-1}, \; \textnormal{for} \; a \rightarrow \infty.
\end{equation}
Let $S_T$ be defined by (\ref{st}). Then,
$$\displaystyle d\left(\frac{S_T-\E(S_T)}{\sqrt{T^d}}, \mathcal{N}(0,\sigma^2)\right) \leq C (\log T)^{-1/4},$$
where $C$ is a constant depending on the field and the level, and $\sigma^2$ satisfies (\ref{sig}).
\end{theorem}
Note that the condition (\ref{dk}) is weak, for example if 
$$\rho(t)\cong (const) \|t\|^{-\alpha}, \quad t \rightarrow +\infty $$
for some positive $\alpha >d$ , then it is met. Here and in the following, the notation $f(x)\cong g(x), \; x\rightarrow a$ means that $\displaystyle \underset{x\rightarrow a}{\lim}\frac{f(x)}{g(x)}= 1$.
\begin{theorem}[\textbf{Moving level}] \label{th3}
Let $\{X(t): \, t\in \R^d \}$ be a random field satisfying the condition (A). Suppose that there exists a positive constant $\alpha \in ]0;2]$ such that in a neighborhood of $0$, the covariance function $\rho$ satisfies
$$1-\rho(t) \cong (const)\|t\|^{\alpha} \; \textnormal{for} \;\; t\rightarrow 0.$$
Let $u_T$ be a function that tends to infinity. One defines the sojourn time as
$$S_T=\int_{[0,T]^d} \mathbb{I}(X(t)\geq u_T)dt.$$
Then, for every $\beta \in (0;d/2)$, there exists a constant $C_{\beta}$ depending on the field such that
$$d\left(\frac{S_T-\E(S_T)}{\sqrt{\textnormal{Var}(S_T)}}, \mathcal{N}(0,1)\right)\leq C_{\beta}\left[\sqrt{\frac{u_T^{\frac{2+\alpha}{\alpha}}}{(\log T)^{1/6}}} + \frac{1}{T^{\beta}\varphi(u_T)u_T}\right].$$
\end{theorem}

In Nourdin et al \cite{14}, the authors consider a very general case of Theorem \ref{th2} in the discrete time and obtain the bound under the form of an optimization problem. Here, in our particular case, we deal with a continuous time field and give an explicit bound.
\section{Preliminaries}
 In this paper, we use some notations that come from the Malliavin calculus introduced as follows.
 \begin{itemize} 
 \item  \textit{Isonormal Gaussian process} \\
Let $\mathfrak{H}$ be a real separable Hilbert space. Denote by $X= \{ \overline{X}(h): \; h \in \mathfrak{H} \}$ an isonormal Gaussian process over $\mathfrak{H}$, that is a centered Gaussian family, defined on some probability space $( \Omega, \mathcal{F}, \PP )$, and $\E (X(h)X(g)) = \langle h,g \rangle_{\mathfrak{H}}$ for every $h, \, g \in \mathfrak{H}$. We assume that $\mathcal{F}$ is generated by $X$.
\item  \textit{Wiener chaos expansion} \\
The $n$-th Hermite polynomial is
$$H_n(x)=(-1)^ne^{\frac{x^2}{2}} \frac{d^n}{dx^n}e^{\frac{-x^2}{2}}.$$
For every $n \geq 1$, the $n$-th Wiener chaos $\mathcal{H}_n$ is defined as the closed linear subspace of $L^2(\Omega,\mathcal{F},\PP)$ generated by the random variables of the type $H_n(X(h))$, where $h \in \mathfrak{H}$ is such that $\|h\|_{\mathfrak{H}}=1$. Then, every square-integrable random variable $Z \in ( \Omega, \mathcal{F}, \PP )$ has the Wiener chaos expansion
\begin{equation}\label{e1}
\displaystyle Z=\sum_{n=0}^{\infty} J_n(Z),
\end{equation}
where $J_0(Z)=\E(Z)$ and $J_n(Z)$ is the projection of $Z$ on $\mathcal{H}_n$. Besides, for any $n \geq 1$ and $h \in \mathfrak{H}, \; \|h\|_\mathfrak{H}=1$, the application
$$I_n(h^{\otimes n})= H_n(X(h)),$$
can be extended to a linear isometry between the symmetric tensor product $\mathfrak{H}^{\odot n}$ equipped with the norm $\sqrt{n!}\|.\|_{\mathfrak{H}^{\otimes n}}$ and the $n$-th Wiener chaos $\mathcal{H}_n$. So, $Z$ can be also decomposed in the form
$$Z=\sum_{n=0}^{\infty} I_n(f_n),$$
where $I_0(c)=c$ for all real $c$, $f_0=\E(Z)$ and $f_n \in \mathfrak{H}^{\odot n}, \, n\geq 1$, are uniquely determined.
\item \textit{Contraction and multiplication}\\
Let $\{e_k, \; k\geq 1 \}$ be a complete orthonormal system in $\mathfrak{H}$. Given $f \in \mathfrak{H}^{\odot p}$ and $g \in \mathfrak{H}^{\odot q}$, then for every $r=0, 1,\ldots,p\wedge q$, the contraction of $f$ and $g$ of order $r$ is the element of $\mathfrak{H}^{\otimes (p+q-2r)}$ defined by
$$f \otimes _r g= \sum^{\infty}_{i_1,\ldots,i_r=1} \langle f, \, e_{i_1}\otimes \ldots \otimes e_{i_r}\rangle_{\mathfrak{H}^{\otimes r}} \otimes \langle g, \, e_{i_1}\otimes \ldots \otimes e_{i_r}\rangle_{\mathfrak{H}^{\otimes r}}. $$
Then, 
$$I_p(f)I_q(g)= \sum_{r=0}^{p\wedge q} r! \binom{p}{r} \binom{q}{r} I_{p+q-2r}(f\widetilde{\otimes}_r g), $$
where $f\widetilde{\otimes}_r g \in \mathfrak{H}^{\odot (p+q-2r)}$ is the symmetrization of $f \otimes _r g$.
\item \textit{Malliavin derivatives}\\
Let $Z$ be a random variable of the smooth form
$$Z=g(X(h_1),\ldots,X(h_n)),$$
where $n\geq 1$, $g: \R ^n \rightarrow \R$ is an infinitely differentiable function with compact support and $h_i \in \mathfrak{H}$. Then, the Malliavin derivative of $Z$ is the element of $L^2(\Omega,\mathfrak{H})$ defined as
$$DZ=\sum_{i=1}^n \frac{\partial g}{\partial x_i}(X(h_1),\ldots,X(h_n))h_i.$$
\item \textit{Ornstein-Uhlenbeck operators} \\
The operator $L$ is defined as $L=\sum_{n=0}^{\infty}-nJ_n$. The domain of $L$ is
$$Dom L= \{Z\in L^2(\Omega): \; \sum_{n=1}^{\infty} n ^2\|J_n(Z)\|^2_2 < \infty\},$$
where $\|J_n(Z)\|_2=\|J_n(Z)\|_{L^2(\Omega)}$. Define the operator $L^{-1}$, called the pseudo-inverse of $L$, as $\displaystyle L^{-1}(Z)=\sum_{n=1}^{\infty}-\frac{1}{n}J_n(Z)$ for all $Z \in L^2(\Omega)$.
\end{itemize}
\section{The fixed level case}
\begin{lemma}\label{lem} For every $n\geq 2$, let $F_n$ be
$$F_n=\frac{1}{\sqrt{T^d}}\int_{[0,T]^d} H_n(X(t))dt.$$
Then,
$$\textnormal{Var}(\|DF_n\|^2_{\mathfrak{H}})\leq \frac{n^4}{T^d} \sum_{r=0}^{n-2}(r!)^2 \binom{n-1}{r}^4(2n-2-2r)! \left(\int_{\R^d}|\rho(t)|dt\right)^3.$$
\end{lemma}
\begin{proof} The Malliavin derivative of $F_n$ is
$$DF_n=\frac{1}{\sqrt{T^d}}\int_{[0,T]^d} n H_{n-1}(X(t)) \psi_t dt,$$
where $\psi_t$ is the element in $\mathfrak{H}$ corresponding to $X(t)$, i.e, $X(t)=\overline{X}(\psi(t))$. And,
$$\|DF_n\|^2_{\mathfrak{H}}=\frac{1}{T^d} n^2 \int_{[0,T]^d\times [0,T]^d} \rho(t-s) H_{n-1}(X(t))H_{n-1}(X(s))dtds.$$
From the Mehler's formula, it is clear that
$$\E[\|DF_n\|^2_{\mathfrak{H}}]=\frac{1}{T^d}n^2 \int_{[0,T]^d\times [0,T]^d} (n-1)!\rho^n(t-s)\,dt\,ds= n\textnormal{Var}(F_n).$$
Using the fact that
$$H_{n-1}(X(t))=I_{n-1}(\psi_t^{\otimes n-1}), $$
and 
\begin{displaymath}
\begin{array}{rl}
I_{n-1}(\psi_t^{\otimes n-1})I_{n-1}(\psi_s^{\otimes n-1}) &=\displaystyle \sum_{r=0}^{n-1} r! \binom{n-1}{r}^2  I_{2n-2-2r}(\psi_t^{\otimes n-1}\widetilde{\otimes}_r \psi_s^{\otimes n-1})\\
&= \displaystyle \sum_{r=0}^{n-1} r! \binom{n-1}{r}^2  \rho^r(t-s) I_{2n-2-2r}(\psi_t^{\otimes n-1-r} \widetilde{\otimes} \psi_s^{\otimes n-1-r}),\\
\end{array}
\end{displaymath}
$\|DF_T\|^2_{\mathfrak{H}}$ can be expressed as
$$\|DF_T\|^2_{\mathfrak{H}}=\frac{1}{T^d} \sum_{r=0}^{n-1} \int_{[0,T]^d\times [0,T]^d} n^2 r! \binom{n-1}{r}^2  \rho^{r+1}(t-s) I_{2n-2-2r}(\psi_t^{\otimes n-1-r} \widetilde{\otimes} \psi_s^{\otimes n-1-r})\, dt\,ds.$$
And, from the orthogonality of the chaos, the variance of $\|DF_T\|^2_{\mathfrak{H}}$ is equal to
\begin{displaymath}
\begin{array}{rl}
 \displaystyle \frac{n^4}{T^{2d}} \sum_{r=0}^{n-2} \int_{[0,T]^{d \times 4}} &(r!)^2 \binom{n-1}{r}^4\rho^{r+1}(t-s) \rho^{r+1}(t'-s')\\
& \displaystyle \times (2n-2-2r)! \langle \psi_t^{\otimes n-1-r} \widetilde{\otimes} \psi_s^{\otimes n-1-r}, \psi_{t'}^{\otimes n-1-r} \widetilde{\otimes} \psi_{s'}^{\otimes n-1-r}\rangle \,dtdsdt'ds'.\\
\end{array}
\end{displaymath}
  Each element of the scalar product has the form
$$\rho^{n-1-r-i}(t-t') \rho^{n-1-r-i}(s-s')\rho^{i}(t-s') \rho^{i}(s-t'),$$
for some $i \in [0; n-1-r]$. And
$$\int_{[0,T]^{d\times 4}} \rho^{r+1}(t-s) \rho^{r+1}(t'-s') \rho^{n-1-r-i}(t-t') \rho^{n-1-r-i}(s-s')\rho^{i}(t-s') \rho^{i}(s-t')\, dtdsdt'ds'$$ is at most equal to
$$\int_{[0,T]^{d\times 4}} |\rho(t-s) \rho(t'-s') \rho(t-t') \rho(s-s')|\, dtdsdt'ds',$$
or
$$\int_{[0,T]^{d\times 4}} |\rho(t-s) \rho(t'-s') \rho(t-s') \rho(s-t')|\, dtdsdt'ds'.$$
With the change of variable $y=(t-s,t'-s',t-t',s')$, $$\int_{[0,T]^{d\times 4}} |\rho(t-s) \rho(t'-s') \rho(t-t') \rho(s-s')|\, dtdsdt'ds'$$ can be written as
$$ \int_{[0,T]^d}dy_4 \int_{A_{y_4}} |\rho(y_1)\rho(y_2)\rho(y_3)\rho(y_2+y_3-y_1)| \, dy_1dy_2dy_3,$$
where $A_{y_4}$ is some domain in $\R^3$ that depends on $y_4$. It is at most equal to
$$\int_{[0,T]^d}dy_4 \int_{\R^d \times \R^d \times \R^d} |\rho(y_1)\rho(y_2)\rho(y_3)| \, dy_1dy_2dy_3=T^d \left(\int_{\R^d}|\rho(t)|dt \right)^3.$$
The same bound is obtained for the others. So, the variance of  $\|DF_T\|^2_{\mathfrak{H}}$ is at most equal to
$$\frac{n^4}{T^d} \sum_{r=0}^{n-2}(r!)^2 \binom{n-1}{r}^4(2n-2-2r)! \left(\int_{\R^d}|\rho(t)|dt \right)^3.$$
\end{proof}
In this paper, we use some facts about Hermite polynomials (see Szeg\"{o} \cite{20}).
\begin{lemma}
\begin{itemize}
\item[$\bullet$] For a fixed point $u$, there exists a constant $C_u$ such that 
\begin{equation}\label{en1}
e^{-u^2/4}|H_n(u)| \leq C_u  (n/e)^{n/2}\; \forall \;n \in \mathbb{N}.
\end{equation}
\item[$\bullet$] There exists a constant $K$ such that, for all $u,\; n$,
\begin{equation}\label{en2}
\frac{\varphi(u) |H_n(u)|}{\sqrt{n!}}<K.
\end{equation}
\item[$\bullet$] As $n$ tends to infinity,
\begin{equation}\label{en3}
\underset{x\in\R}{\max}\; e^{-x^2/4}|H_n(x)| \cong (const) \sqrt{n!}\,n^{-1/12}. 
\end{equation}
\end{itemize}
\end{lemma}
\begin{proof}[\textbf{Proof of Theorem \ref{th2}}]
It is clear that
\begin{displaymath}
\begin{array}{rl}
 \displaystyle d\left(\frac{S_T-\E(S_T)}{\sqrt{T^d}}, \mathcal{N}(0,\sigma^2)\right) \leq  &d\left(\frac{S_T-\E(S_T)}{\sqrt{T^d}},\frac{S_{T,N_T}-\E(S_{T,N_T})}{\sqrt{T}}\right) \\
&+ d\left(\frac{S_{T,N_T}-\E(S_{T,N_T})}{\sqrt{T^d}},\mathcal{N}(0,\sigma_{N_T}^2) \right)+d\left(\mathcal{N}(0,\sigma_{N_T}^2) ,\mathcal{N}(0,\sigma^2)\right)=d_1+d_2+d_3,\\
\end{array}
\end{displaymath}
where $S_{T,N_T}$ is the truncation of $S_T$ at position $N_T$ in the Wiener chaos expansion. $N_T$ will be chosen later on.
\begin{itemize}
\item[i)]  \textbf{(Bound for $d_1$)} It is easy to show that
\begin{displaymath}
\begin{aligned}
& \displaystyle d\left(\frac{S_T-\E(S_T)}{\sqrt{T^d}},\frac{S_{T,N_T}-\E(S_{T,N_T})}{\sqrt{T^d}}\right)\\
\leq \quad & \displaystyle \left\|\frac{S_T-S_{T,N_T}}{\sqrt{T^d}}\right\| _2 \\
= \quad & \displaystyle \sqrt{\sum_{n=N_T+1}^{\infty}\frac{\varphi^2(u)H_{n-1}^2(u)}{n!T^d} \int_{[-T,T]^d}\rho^n(t)\underset{j=1}{\overset{d}{\Pi}}(T-|t_j|)dt}.
\end{aligned}
\end{displaymath}
Here, from (\ref{en1}) and the Stirling formula
$$n! \sim \sqrt{2\pi n}(n/e)^n,$$
we obtain the bound for $d_1$
\begin{equation}\label{d1}
d_1 \leq \displaystyle C_u \sqrt{\varphi(u)} \sqrt{\int_{\R^d}|\rho(t)|dt}\sqrt{\sum_{n=N_T+1}^{\infty}n^{-(1+\frac{1}{2})}}\leq (const) N_T^{-1/4}.
\end{equation}
\item[ii)] \textbf{(Bound for $d_2$)} From Theorem 3.1 of \cite{12}, it is clear that 

\begin{displaymath}
\begin{array}{rl}
 & \displaystyle d\left(\frac{S_{T,N_T}-\E(S_{T,N_T})}{\sqrt{T^d}},\mathcal{N}(0,\sigma_{N_T}^2) \right)\\
 & \\
 \leq & \displaystyle \left\|\sigma_{N_T}^2-\left\langle D\frac{S_T-\E(S_T)}{\sqrt{T^d}},-DL^{-1}\frac{S_T-\E(S_T)}{\sqrt{T^d}} \right\rangle_{\mathfrak{H}}\right\|_2\\
\leq &  \displaystyle \sum_{p,q=1}^{N_T}\left\|\delta_{pq}\sigma_{T,p}^2-q^{-1}\left\langle DJ_p\left(\frac{S_T-\E(S_T)}{\sqrt{T^d}}\right), DJ_q\left(\frac{S_T-\E(S_T)}{\sqrt{T^d}}\right)\right\rangle_{\mathfrak{H}}\right\|_2,\\
\end{array}
\end{displaymath}
where $J_p$ is the component in the $p$-th chaos defined in (\ref{e1}) and $\sigma_{T,p}^2$ is the variance of $J_p(\frac{S_T-\E(S_T)}{\sqrt{T^d}}).$
\begin{itemize}
\item  If $p=q=1$, $$\displaystyle \left\|\sigma_{T,1}^2-\left\langle DJ_1\left(\frac{S_T-\E(S_T)}{\sqrt{T^d}}\right), DJ_1\left(\frac{S_T-\E(S_T)}{\sqrt{T^d}}\right)\right\rangle_{\mathfrak{H}}\right\|_2=0.$$
\item  If $p=q>1$, 
\begin{displaymath}
\begin{aligned}
& \displaystyle \left\|\sigma_{T,p}^2-p^{-1}\left\langle DJ_p\left(\frac{S_T-\E(S_T)}{\sqrt{T^d}}\right), DJ_p\left(\frac{S_T-\E(S_T)}{\sqrt{T^d}}\right)\right\rangle_{\mathfrak{H}}\right\|_2\\
=& \displaystyle p^{-1}\sqrt{\textnormal{Var}\left(DJ_p\left(\frac{S_T-\E(S_T)}{\sqrt{T^d}}\right)\right)}\\
=& \displaystyle \frac{\varphi^2(u)H^2_{p-1}(u)}{(p!)^2} p^{-1} \sqrt{\textnormal{Var}\left(\left\|D\left(\frac{1}{\sqrt{T^d}} \int_0^T H_p(X(t))dt\right)\right\|^2_{\mathfrak{H}}\right)}.
\end{aligned}
\end{displaymath}
Then, from Lemma \ref{lem}, it is at most equal to
$$\frac{\varphi^2(u)H^2_{p-1}(u)}{(p!)^2}\frac{p}{\sqrt{T^d}}\sqrt{\left(\int_{\R^d}|\rho(t)|dt\right)^3 \left(\sum_{r=0}^{p-2} (r!)^2\binom{p-1}{r}^4(2p-2-2r)!\right)} .$$
\item If $p>1$ and $q=1$, then
\begin{displaymath}
\begin{aligned}
&\displaystyle \left\langle D\left(\frac{1}{\sqrt{T^d}} \int_{[0,T]^d} H_p(X(t))dt\right), D\left(\frac{1}{\sqrt{T^d}} \int_{[0,T]^d} H_1(X(t))dt\right)\right\rangle_{\mathfrak{H}}\\
 =\quad & \displaystyle\frac{p}{T^d} \int_{[0,T]^d \times [0,T]^d} \rho(t-s)H_{p-1}(X(t))dtds.
\end{aligned}
\end{displaymath}
So, its variance is 
$$ \frac{1}{T^{2d}}p^2 \int_{[0,T]^{d\times 4}} (p-1)! \rho(t-s) \rho(t'-s') \rho^{p-1}(t-t')dtdsdt'ds' \leq \frac{p^2(p-1)!}{T^d} \left(\int_{\R^d}|\rho(t)|dt\right)^3.$$
\item If $p, \, q >1$ and $p\neq q$, then 
\begin{displaymath}
\begin{array}{rl}
&\displaystyle \left\langle D\left(\frac{1}{\sqrt{T^d}} \int_0^T H_p(X(t))dt\right), D\left(\frac{1}{\sqrt{T^d}} \int_0^T H_q(X(t))dt\right)\right\rangle_{\mathfrak{H}}\\
=&\displaystyle \frac{pq}{T^d} \int_{[0,T]^d \times [0,T]^d} \rho(t-s)H_{p-1}(X(t))H_{q-1}(X(s))dtds\\
=&\displaystyle \frac{pq}{T^d}  \sum_{r=0}^{p\wedge q -1} \int_{[0,T]^d \times [0,T]^d}r!\binom{p-1}{r}\binom{q-1}{r}\rho(t-s) I_{p+q-2-2r}((\psi_t^{\otimes p-1} \widetilde{\otimes}_r  \psi_s^{\otimes q-1})_s)dtds.
\end{array}
\end{displaymath}
 So, its variance is at most equal to
$$\leq \frac{(pq)^2}{T^d}  \left(\int_{\R^d}|\rho(t)|dt\right)^3\left( \sum_{r=0}^{p\wedge q -1} (r!)^2\binom{p-1}{r}^2\binom{q-1}{r}^2 (p+q-2-2r)!\right).$$
\end{itemize}
We obtain the bound for $d_2$
$$ \sqrt{\left(\int_{\R^d}|\rho(t)|dt\right)^3}\left[ \sum_{p=2}^{N_T} \frac{\varphi^2(u)H^2_{p-1}(u)}{(p!)^2}\frac{p}{\sqrt{T^d}} \sqrt{\sum_{r=0}^{p-2} (r!)^2\binom{p-1}{r}^4(2p-2-2r)!} \right.$$
$$+  \left. \sum_{p,q=1; p\neq q}^{N_T} \left(\frac{1}{p}+\frac{1}{q}\right) \frac{\varphi^2(u)|H_{p-1}(u)H_{q-1}(u)|}{p!q!} \frac{pq}{\sqrt{T^d}}\sqrt{ \sum_{r=0}^{p\wedge q -1} (r!)^2\binom{p-1}{r}^2\binom{q-1}{r}^2 (p+q-2-2r)!} \right].$$
So,
\begin{equation}\label{d2}
\displaystyle d_2 \leq (const)\frac{3^{N_T}}{\sqrt{T^d}}.
\end{equation}
Indeed, from
\begin{displaymath}
\begin{array}{rl}
& \displaystyle \frac{1}{((p-1)!)^2}\sum_{r=0}^{p-2} (r!)^2\binom{p-1}{r}^4(2p-2-2r)!\\
=& \displaystyle \sum_{r=0}^{p-2} \binom{p-1}{r}^2 \binom{2p-2-2r}{p-1-r}\\
\leq & \displaystyle \sum_{r=0}^{p-2} \binom{p-1}{r}^2 2^{2p-2-2r}\\
\leq & \displaystyle 2^{2p-2} \left(\sum_{r=0}^{p-2} \binom{p-1}{r}2^{-r}\right)^2\\
\leq &  2^{2p-2} (1+1/2)^{2p-2}=9^{p-1}, 
\end{array}
\end{displaymath}
and (\ref{en2}), the first term is at most equal to
 $$(const) \frac{1}{\sqrt{T^d}} \sum_{p=2}^{N_T} \frac{3^{p-1}}{p}; $$
and the same for the second term.
\item[iii)] \textbf{(Bound for $d_3$)} It is easy to show that 
\begin{displaymath}
\begin{array}{rl}
d^2(\mathcal{N}(0,\sigma_{N_T}^2) ,\mathcal{N}(0,\sigma^2)) \leq & (const) (\sigma^2-\sigma_{N_T}^2)\\
  =&\left[ \displaystyle \sum_{n=1}^{N_T}\frac{\varphi^2(u)H_{n-1}^2(u)}{n!} \int_{[-T,T]^d}\rho^n(t)\frac{T^d-\underset{j=1}{\overset{d}{\Pi}}(T-|t_j|)}{T^d}dt \right. \\
&+\quad \displaystyle \sum_{n=1}^{N_T}\frac{\varphi^2(u) H_{n-1}^2(u)}{n!} \int_{\R^d\setminus [-T,T]^d}\rho^n(t)dt\\
 &+  \left. \displaystyle \sum_{n=N_T+1}^{\infty}\frac{\varphi^2(u)H_{n-1}^2(u)}{n!} \int_{\R^d}\rho^n(t)dt \right].\\ 
\end{array}
\end{displaymath}
From part i), the third term is at most equal to $(const) N_T^{-1/2}$. 
For the first term, it is equal to
\begin{displaymath}
\begin{array}{rl}
& \displaystyle \sum_{n=1}^{N_T}\frac{\varphi^2(u)H_{n-1}^2(u)}{n!} \int_{[-\sqrt{T},\sqrt{T}]^d}\rho^n(t)\frac{T^d-\underset{j=1}{\overset{d}{\Pi}}(T-|t_j|)}{T^d}dt\\
+& \displaystyle \sum_{n=1}^{N_T}\frac{\varphi^2(u)H_{n-1}^2(u)}{n!} \int_{[-T,T]^d \setminus [-\sqrt{T},\sqrt{T}]^d}\rho^n(t)\frac{T^d-\underset{j=1}{\overset{d}{\Pi}}(T-|t_j|)}{T^d}dt,\\
\end{array}
\end{displaymath}
which is at most equal to
\begin{displaymath}
\begin{array}{rl}
 & \displaystyle \sum_{n=1}^{N_T}\frac{\varphi^2(u)H_{n-1}^2(u)}{n!\sqrt{T}} \int_{[-\sqrt{T},\sqrt{T}]^d}|\rho^n(t)|dt\\
+ & \displaystyle  \sum_{n=1}^{N_T}\frac{\varphi^2(u)H_{n-1}^2(u)}{n!} \int_{[-T,T]^d \setminus [-\sqrt{T},\sqrt{T}]^d}|\rho^n(t)|dt.\\
\end{array}
\end{displaymath}
The first part is at most equal to $\displaystyle \frac{(const)}{\sqrt{T}}$. The sum of the second part and the second term is
$$\sum_{n=1}^{N_T}\frac{\varphi^2(u)H_{n-1}^2(u)}{n!} \int_{\R^d \setminus [-\sqrt{T},\sqrt{T}]^d}|\rho^n(t)|dt,$$
and at most equal to $(const)(\log T)^{-1}$ (from (\ref{dk})). So, 
\begin{equation}\label{d3}
d_3 \leq (const)(N_T^{-1/4}+T^{-1/4}+(\log T)^{-1/2}).
\end{equation}
\end{itemize}
Summing up three bounds (\ref{d1}), (\ref{d2}) and (\ref{d3}), by choosing $N_T=(\log T)/4 $, we have the result.
\end{proof}
\section{The moving level case}
In this section, we assume that the level depends on $T$ and we denote by $u_T$. Then the sojourn time
$$S_T=\int_{[0,T]^d} \mathbb{I}(X(t)\geq u_T)dt$$
has $$\E(S_T)=T^d\overline{\Phi}(u_T) $$ and 
$$\textnormal{Var}(S_T)=\int_{[-T,T]^d}\underset{j=1}{\overset{d}{\Pi}}(T-|t_j|)dt \int_0^{\rho(t)} \varphi(u_T,u_T,y)dy ,$$
where 
$$\varphi(u_T,u_T,y)=\frac{1}{2\pi \sqrt{1-y^2}}\exp\left(\frac{-u_T^2}{1+y}\right)$$ is the density of the bivariate normal vector 
$$\mathcal{N}\left(0, \left[\begin{array}{cc} 1& y\\ y & 1  \end{array}\right] \right).$$
When $u_T$ tends to infinity, 
$$\frac{\textnormal{Var}(S_T)}{T^d} \rightarrow 0,$$
then the Theorem \ref{th1} and \ref{th2} no longer hold. So, at first, we generalize the results of Berman \cite{3} (chapter 8) to estimate the variance of $S_T$ (the detailed proofs are given in the Appendix).
\begin{lemma} If the covariance function $\rho$ satisfies the conditions in Theorem \ref{th3}, then, for every $\epsilon >0$,
$$\int_{[-T,T]^d}\underset{j=1}{\overset{d}{\Pi}}(T-|t_j|)dt \int_0^{\rho(t)} \varphi(u_T,u_T,y)dy  \cong T^d \int_{[-\epsilon,\epsilon]^d} \int_0^{\rho(t)} \varphi(u_T,u_T,y)dy dt,$$
for $T,u_T \rightarrow \infty$.
\end{lemma}

So, let $B(u)$ be some function that satisfies
$$B(u) \cong \int_{[-\epsilon,\epsilon]^d} \int_0^{\rho(t)} \varphi(u,u,y)dy dt, \; \textnormal{for} \;\; u \rightarrow  \infty.$$
Then,
$$\textnormal{Var}(S_T) \cong T^d B(u_T),$$
for $T,u_T \rightarrow \infty$.
\begin{lemma}\label{lem1}
If the covariance function $\rho$ satisfies the conditions in Theorem \ref{th3}, then,
 $$B(u)\cong (const) \frac{ \varphi(u)}{u^{\frac{2+\alpha}{\alpha}}}, \; \textnormal{for} \;\; u \rightarrow \infty.$$ 
\end{lemma}

\begin{proof}[\textbf{Proof of Theorem \ref{th3}.}]
The distance between $\displaystyle \frac{S_T-\E(S_T)}{\sqrt{\textnormal{Var}(S_T)}}$ and the standard Gaussian variable is at most equal to 
$$d\left(\frac{S_T-\E(S_T)}{\sqrt{\textnormal{Var}(S_T)}},\frac {S_{T,N_T}-\E(S_{T,N_T})}{\sqrt{\textnormal{Var}(S_{T,N_T})}}\right)+d\left(\frac{S_{T,N_T}-\E(S_{T,N_T})}{\sqrt{\textnormal{Var}(S_{T,N_T})}},\mathcal{N}(0,1)\right),$$
where $S_{T,N_T}$ is the truncate variable of $S_T$ at position $N_T$ in the Wiener chaos expansion. $N_T$ will be chosen later on.
\begin{itemize}
\item The first term is at most equal to (up to some multiplicative constants)
\begin{displaymath}
\begin{aligned}
&\displaystyle \sqrt{\frac{\textnormal{Var}(S_T-S_{N_T})}{\textnormal{Var}(S_T)}}\\
=\quad & (const) \displaystyle \sqrt{\frac{\sum_{n=N_T+1}^{\infty}\frac{\varphi^2(u_T)H_{n-1}^2(u_T)}{n!} \int_{[-T,T]^d} \rho^n(t) \underset{j=1}{\overset{d}{\Pi}}(T-|t_j|)dt}{\textnormal{Var}(S_T)}}\\
\cong \quad & (const) \displaystyle \sqrt{u_T^{\frac{2+\alpha}{\alpha}}\sum_{n=N_T+1}^{\infty}\frac{\varphi(u_T)H_{n-1}^2(u_T)}{n!} \int_{[-T,T]^d} \rho^n(t) \frac{\underset{j=1}{\overset{d}{\Pi}}(T-|t_j|)}{T^d}dt } \\
\leq \quad & \displaystyle (const) \sqrt{u_T^{\frac{2+\alpha}{\alpha}}\int_{\R^d}|\rho(t)|dt \sum_{n=N_T+1}^{\infty} \frac{\varphi(u_T)H_{n-1}^2(u_T)}{n!}}\\
\leq \quad & \displaystyle (const) \sqrt{u_T^{\frac{2+\alpha}{\alpha}}\sum_{n=N_T+1}^{\infty}n^{-(1+\frac{1}{6})}}\quad,
\end{aligned}
\end{displaymath}
where in the third line, we use the approximation
$$\textnormal{Var}(S_T) \cong T^d B(u_T) \cong (const) T^d  \frac{ \varphi(u_T)}{u_T^{\frac{2+\alpha}{\alpha}}},$$
and in the last one, the fact (\ref{en3}) is used. Then, we have the bound
\begin{equation}\label{d4}
\displaystyle (const) \sqrt{\frac{\varphi(u_T) \int_{\R^d}|\rho(t)|dt \sum_{n=N_T+1}^{\infty}n^{-(1+\frac{1}{6})}}{B(u_T)}} \leq (const)
 \sqrt{\frac{u_T^{\frac{2+\alpha}{\alpha}}}{N_T^{1/6}}}.
\end{equation}
 
\item For the second term, as the same argument in part ii) in the proof of Theorem \ref{th2}, we have the bound
$$(const) \frac{3^{N_T}}{\sqrt{T^d}\sqrt{\sum_{n=1}^{N_T}\frac{\varphi^2(u_T)H_{n-1}^2(u_T)}{n!} 2\int_{[-T,T]^d}\rho^n(t)\frac{\underset{j=1}{\overset{d}{\Pi}}(T-|t_j|)}{T^d}dt}},$$
which is at most equal to
\begin{equation}\label{d5}
\displaystyle \frac{3^{N_T}}{T^{d/2}\varphi(u_T)u_T}.
\end{equation}
\end{itemize}
Summing up (\ref{d4}) and (\ref{d5}), by choosing $N_T$ such that $3^{N_T}=T^{-\beta+d/2}$, the result follows.
\end{proof}
We have the following corollary
\begin{corollary}
Let $\{X(t): \, t\in \R^d \}$ be a random field satisfying the condition (A). Suppose that there exists a positive constant $\alpha \in ]0;2]$ such that in a neighborhood of $0$, the covariance function $\rho$ satisfies
$$1-\rho(t) \cong (const)\|t\|^{\alpha} \; \textnormal{for} \;\; t\rightarrow 0.$$
One defines the sojourn time
$$S_T=\int_{[0,T]^d} \mathbb{I}(X(t)\geq u_T)dt.$$
Let $u_T$ be a function that tends to infinity. Then, if $$ (\log T)^{-1/6}u_T^{\frac{2+\alpha}{\alpha}} \rightarrow 0, $$ 
one has
$$\frac{S_T-\E(S_T)}{\sqrt{\textnormal{Var}(S_T)}} \xrightarrow{d} \mathcal{N}(0,1).$$
\end{corollary}
\begin{proof} Since $ (\log T)^{-1/6}u_T^{\frac{2+\alpha}{\alpha}} \rightarrow 0$, it is easy to see that
$$\frac{1}{T^{\beta}\varphi(u_T)u_T} \rightarrow 0,$$
for all $\beta \in (0;d/2)$ .From Theorem \ref{th3}, the result follows.
\end{proof}
This extends, under the stronger hypothesis on $u_T$, the results of Berman to Gaussian fields in $\R^d$ with $d>1$.\\
\section*{Appendix: Proofs of the Lemmas 3-4}
In this Appendix, we prove the Lemmas 3-4 analogously to the similar ones in \cite{3} with some minor changes.
\begin{proof}[\textbf{Proof of Lemma 3.}]
It suffices to show that
\begin{equation}\label{dg1}
  \frac{\int_{[-T,T]^d \setminus [-\epsilon,\epsilon]^d}\underset{j=1}{\overset{d}{\Pi}}(T-|t_j|) \int_0^{\rho(t)} \phi (u_T,u_T,y) \,dy dt}{\int_{[-\epsilon,\epsilon]^d}\underset{j=1}{\overset{d}{\Pi}}(T-|t_j|) \int_0^{\rho(t)} \phi(u_T,u_T,y)dy dt}
\end{equation}
tends to $0$ for $u_T,T \rightarrow \infty$. 
In fact, denote $$\eta= 1-\max(|\rho(s)| \; : \; s \notin (-\epsilon,\epsilon)^d).$$
If $\eta=0$ then there exists $x\neq 0 $ such that $|\rho(x)|=1$, then the field is $x$- or $2x$- periodic and the integral $\displaystyle \int_{\R^d}|\rho(t)|dt$ can not converge. Therefore, $\eta$ is strictly positive. Since the function $\varphi(u_T,u_T,y)$ is increasing with respect to $y$, the numerator in (\ref{dg1}) is at most equal to
\begin{equation}\label{t1}
\displaystyle T^d\varphi(u_T,u_T,1-\eta)\int_{[-T,T]^d\setminus [-\epsilon,\epsilon]^d}|\rho(t)|dt.
\end{equation}
The denominator in (\ref{dg1}) can be decomposed as
\begin{equation}\label{dg2}
\displaystyle \int_{[-\epsilon,\epsilon]^d}\underset{j=1}{\overset{d}{\Pi}}(T-|t_j|) \int_0^{\rho(t)^+} \varphi(u_T,u_T,y)dy dt-\int_{[-\epsilon,\epsilon]^d}\underset{j=1}{\overset{d}{\Pi}}(T-|t_j|) \int_{-\rho(t)^-}^0 \varphi(u_T,u_T,y)dy dt.
\end{equation}
There exists a positive constant $c <1$, such that $\rho(t)^-\leq c, \; \forall t$, then the second term in (\ref{dg2}) is at most equal to
\begin{equation}\label{t2}
\displaystyle (2\epsilon)^d T^d \frac{1}{\sqrt{1-c^2}}\varphi^2(u_T).
\end{equation}
Choose $\delta < \eta$ and $\epsilon'<\epsilon$ such that
$$\min(\rho(t)\; : \; t \in [-\epsilon',\epsilon']^d) \geq 1-\delta,$$
then the first term in (\ref{dg2}) is lower-bounded by
\begin{displaymath}
\begin{aligned}
& \displaystyle \int_{[-\epsilon',\epsilon']^d}\underset{j=1}{\overset{d}{\Pi}}(T-|t_j|) \int_0^{\rho(t)} \varphi(u_T,u_T,y)dy dt \nonumber \\
\geq \quad &\displaystyle (T-\epsilon')^d\int_{[-\epsilon',\epsilon']^d} \int_{1-\delta}^{\rho(t)} \varphi(u_T,u_T,y)dy dt. \nonumber \\
\end{aligned}
\end{displaymath}
and it has the lower bound
\begin{equation}\label{dg3}
 \displaystyle (T-\epsilon')^d \varphi(u_T,u_T,1-\delta) \int_{[-\epsilon',\epsilon']^d} (\rho(t)-1+\delta)dt.
\end{equation}
It is clear that (\ref{t1}) and (\ref{t2}) are negligible with respect to (\ref{dg3}) when $u_T$ and $T$ tend to infinity. it implies the result.
\end{proof}
To prove the lemma \ref{lem1}, we need the following two results:
\begin{lemma}\label{lem2}
For every $\theta >1$, there exists a constant $K(\theta)>0$, such that, asymptotically
$$B(u) \geq K(\theta)\exp(-u^2\theta/2).$$
\end{lemma}
\begin{proof}
It suffices to prove the lemma for $\theta$ in a neighborhood of $1$. In such case, using (\ref{dg3}), we can choose $\delta$ such that
$$\exp(-u^2\theta/2)=\varphi(u,u,1-\delta),$$
and we are done.
\end{proof}
\begin{lemma}\label{lem3}
For every $\delta \in (0,1)$, one has
$$\underset{u\rightarrow \infty}{\lim\sup}\; \frac{B(u)}{2(\frac{2}{2-\delta})^{1/2} \frac{\varphi(u)}{u} \int_{[- \epsilon,\epsilon]^d} \overline{\Phi}\left(u\left[\frac{1-\rho(t)}{2}\right]^{1/2}\right)dt} \leq 1,$$
and 
$$\underset{u\rightarrow \infty}{\lim\inf}\; \frac{B(u)}{[2(2-\delta)]^{1/2} \frac{\varphi(u)}{u} \int_{[- \epsilon,\epsilon]^d} \overline{\Phi}\left(u\left[\frac{1-\rho(t)}{2-\delta}\right]^{1/2}\right)dt}\geq 1.$$
\end{lemma}
\begin{proof}
For any $\delta \in (0,1)$, there exists $\epsilon >0$ such that $1-\rho(s) < \delta, \; \forall s \in [-\epsilon,\epsilon]^d$. Then,
$$\int_{[- \epsilon,\epsilon]^d} \int_0^{1-\delta}\varphi(u,u,y)dydt=(2\epsilon)^d\int_0^{1-\delta}\varphi(u,u,y)dy \leq (2\epsilon)^d (1-\delta) \varphi(u,u,1-\delta).$$
Since
$$\varphi(u,u,1-\delta)=\frac{1}{2\pi \sqrt{1-(1-\delta)^2}} \exp\left(\frac{-u^2}{2-\delta}\right),$$
and from Lemma \ref{lem2}, $B(u)$ is asymptotically greater than $K(\theta)\exp(-u^2\theta/2)$ for every $\theta >1$, then by choosing
$$1<\theta< \frac{2}{2-\delta},$$
$\int_{[- \epsilon,\epsilon]^d} \int_0^{1-\delta}\varphi(u,u,y)dydt$ is negligible with respect to $B(u)$ when $u$ tends to infinity. Hence, $B(u)$ is asymptotically equal to
\begin{displaymath}
\begin{aligned}
& \displaystyle \int_{[-\epsilon,\epsilon]^d} \int_{1-\delta}^{\rho(t)} \varphi(u,u,y)dy dt \\
=&\displaystyle \varphi(u) \int_{[-\epsilon,\epsilon]^d} \int_{1-\delta}^{\rho(t)} \frac{1}{\sqrt{1-y^2}} \varphi\left( u\left[\frac{1-y}{1+y}\right]^{1/2}\right)dydt,
\end{aligned}
\end{displaymath}
which is equal to, by the change of variable $z=u^2(1-y)$,
\begin{equation}\label{dg4}
\displaystyle \frac{\varphi(u)}{u} \int_{[-\epsilon,\epsilon]^d} \int_{u^2(1-\rho(t))}^{u^2\delta} \frac{1}{\sqrt{z(2-z/u^2)}} \varphi\left( \left[\frac{z}{2-z/u^2}\right]^{1/2}\right)dzdt.
\end{equation}
An upper bound of (\ref{dg4}) is
$$\frac{1}{\sqrt{2-\delta}}\frac{\varphi(u)}{u} \int_{[-\epsilon,\epsilon]^d} \int_{u^2(1-\rho(t))}^{\infty}\varphi(\sqrt{z/2})\frac{dz}{\sqrt{z}}dt,$$
which is equal to, by the change of variable $x=\sqrt{z/2}$,
\begin{displaymath}
\begin{aligned}
& \displaystyle \frac{2\sqrt{2}}{\sqrt{2-\delta}}\frac{\varphi(u)}{u} \int_{[-\epsilon,\epsilon]^d} \int_{u^2(1-\rho(t))}^{\infty}\varphi(x)dxdt\\
=& \displaystyle \frac{2\sqrt{2}}{\sqrt{2-\delta}}\frac{\varphi(u)}{u} \int_{[-\epsilon,\epsilon]^d} \overline{\Phi}\left(u\left[\frac{1-\rho(t)}{2}\right]^{1/2}\right)dt.
\end{aligned}
\end{displaymath}
A lower bound of (\ref{dg4}) is
$$\frac{\varphi(u)}{u} \int_{[-\epsilon,\epsilon]^d} \int_{u^2(1-\rho(t))}^{u^2\delta}\varphi\left( \left[\frac{z}{2-\delta}\right]^{1/2}\right)\frac{dz}{\sqrt{2z}}dt,$$
which is equal to, by the change of variable $x=\sqrt{z/(2-\delta)}$,
\begin{displaymath}
\begin{aligned}
& \displaystyle \sqrt{2(2-\delta)}\frac{\varphi(u)}{u} \int_{[-\epsilon,\epsilon]^d} \int_{u(\frac{1-\rho(t)}{2-\delta})^{1/2}}^{u(\delta/(2-\delta))^{1/2}}\varphi(x)dxdt\\
=& \displaystyle \sqrt{2(2-\delta)}\frac{\varphi(u)}{u} \int_{[-\epsilon,\epsilon]^d} \left[ \overline{\Phi}\left(u\left[\frac{1-\rho(t)}{2-\delta}\right]^{1/2}\right)-\overline{\Phi}\left(u\left[\frac{\delta}{2-\delta}\right]^{1/2}\right) \right]dt.
\end{aligned}
\end{displaymath}
Since
$$\underset{u\rightarrow \infty}{\lim} \underset{s\in [-\epsilon,\epsilon]^d}{\sup} \frac{\overline{\Phi}\left(u\left[\frac{1-\rho(t)}{2-\delta}\right]^{1/2}\right)}{\overline{\Phi}\left(u\left[\frac{\delta}{2-\delta}\right]^{1/2}\right)}=0,$$
the lower bound is asymptotically equal to
$$[2(2-\delta)]^{1/2} \frac{\varphi(u)}{u} \int_{[- \epsilon,\epsilon]^d} \overline{\Phi}\left(u\left[\frac{1-\rho(t)}{2-\delta}\right]^{1/2}\right)dt.$$
\end{proof}

\begin{proof}[\textbf{Proof of Lemma \ref{lem1}.}]
By change of variable $t=z/u^{2/\alpha}$, the asymptotically upper bound in Lemma \ref{lem3} is equal to
$$2\left(\frac{2}{2-\delta}\right)^{1/2} \frac{\varphi(u)}{u^{\frac{2+\alpha}{\alpha}}} \int_{[- u^{2/\alpha}\epsilon,u^{2/\alpha}\epsilon]^d}\overline{\Phi}\left(u\left[\frac{1-\rho(z/u^{2/\alpha})}{2}\right]^{1/2}\right)dz.$$
It is clear that
$$u^2(1-\rho(z/u^{2/\alpha})) \rightarrow C\|z\|^{\alpha} \; \textnormal{for} \;\; u \rightarrow \infty,$$
then by dominated convergence, this upper bound is asymptotically equal to
$$2\left(\frac{2}{2-\delta}\right)^{1/2} \frac{\varphi(u)}{u^{\frac{2+\alpha}{\alpha}}} \int_{\R^d}\overline{\Phi}(C\|z\|^{\alpha})dz.$$
By the same argument, the lower one in Lemma \ref{lem3} is asymptotically equal to
$$[2(2-\delta)]^{1/2} \frac{\varphi(u)}{u^{\frac{2+\alpha}{\alpha}}} \int_{\R^d}\overline{\Phi}(C\|z\|^{\alpha})dz.$$
Let $\delta$ tend to $0$, we obtain the result.
\end{proof}
\textbf{Acknowledgements:} I would like to thank Jean-Marc Aza\"{i}s and Jos\'{e} Rafael Le\'{o}n for their invaluable discussions and unconditional support. I also thank two anonymous reviewers for their constructive remarks.
 
\end{document}